\documentclass[11pt]{amsart}

\usepackage{amssymb,verbatim,url}
\usepackage{graphicx,pdfsync}
\usepackage{amssymb}
 \usepackage{amsthm}
\usepackage{amsmath,verbatim}

\usepackage{color}

\newtheorem{theorem}{Theorem}[section]
\newtheorem{corollary}[theorem]{Corollary}

\newtheorem{proposition}[theorem]{Proposition}
\theoremstyle{definition}

\newcommand{\I}{\mathrm{i}}
\newcommand{\D}{\mathrm{d}}
\newcommand{\vp}{\varphi}
\newcommand{\A}{\alpha}
\newcommand{\B}{\beta}
\newcommand{\lb}{\left(}
\newcommand{\rb}{\right)}

\newcommand{\Rb}{\mathbb{R}}

\newcommand{\PD}{\partial}

\newcommand{\wt}{\widetilde}
\newcommand{\Cc}{\mathcal{C}}

\newcommand{\Nb}{\mathbb{N}}

\newcommand{\Beq}{\begin{equation}}
\newcommand{\Eeq}{\end{equation}}
\newcommand{\beq}{\begin{equation*}}
\newcommand{\eeq}{\end{equation*}}
\newcommand{\bal}{\begin{align}}
\newcommand{\eal}{\end{align}}
\renewcommand{\O}{\Omega}

\newcommand{\g}{\gamma}
\newcommand{\bpr}{\begin{proof}}
\newcommand{\epr}{\end{proof}}

\newcommand{\tred}[1]{{\color{red}{#1}}}

\newcommand{\bel}[1]{\begin{equation}\label{#1}}
\newcommand{\ee}{\end{equation}}

\newcommand{\NT}{\negthinspace}

\newcommand{\wh}{\widehat}

\newcommand{\ve}{\varepsilon}

\newcommand{\Bc}{\mathcal{B}}

\newcommand{\Nc}{\mathcal{N}}

\newcommand{\Qc}{\mathcal{Q}}

\newcommand{\Cb}{\mathbb{C}}

\newcommand{\Sb}{\mathbb{S}}

\renewcommand{\O}{\Omega}

\newcommand{\bp}{\begin{prob}}
\newcommand{\ep}{\end{prob}}
\newcommand{\ac}[1]{\begin{quotation}\textbf{Anupam's comment:\
}{\textit{#1}}\end{quotation}}

\begin{document}

%\begin{frontmatter}



\title[Stability estimates for a biharmonic boundary value problem]
{Stability estimates for the inverse boundary value problem for the biharmonic operator with bounded potentials}


\author[Choudhury and Krishnan]{Anupam Pal Choudhury and Venkateswaran P.~Krishnan}

\address{Tata Institute of Fundamental Research, Centre for Applicable Mathematics, Bangalore, India}
\email{anupam@math.tifrbng.res.in, vkrishnan@math.tifrbng.res.in}

\begin{abstract}
In this article, stability estimates are given for the determination of the  zeroth-order bounded perturbations of the biharmonic operator when the boundary Neumann measurements are made on the whole boundary and on slightly more than half the boundary, respectively. For the case of measurements on the whole boundary, the stability estimates are of $\ln$-type and for the case of measurements on slightly more than half of the boundary, we derive estimates that are of $\ln\ln$-type. %The proof follows the methods used to prove stability estimates for the Calder\'on problem with full and slightly more than half partial data.
\end{abstract}

%\begin{keyword}
%Biharmonic equation; stability estimates; inverse problems; Calder\'on problem.

%\MSC[2010] 35J10, 35J40.
%\end{keyword}

%\end{frontmatter}
\maketitle

% \linenumbers

%\section{}
%\label{}

\section{Introduction }
Let $\O \subset \Rb^{n}, n\geq 3$  be a bounded domain with $C^{\infty}$ boundary and consider the following equation:  
\[
\Bc_{q}u:= (\Delta^{2}+q) u=0 \mbox{ in } \O,  \quad q\in L^{\infty}(\O).
\]
 Let the domain of $\Bc_{q}$ be 
 \[
 D(\Bc_{q}):= \{u\in H^{4}(\O): u|_{\PD \O}=\Delta u|_{\PD \O}=0\}.
 \] 
 We consider the following space for the potential $q$: 
 \Beq\label{potential-space}
\Qc_{M}:=\{ q: \mbox{ supp}(q)\subset \overline{\O},\mbox{ and } \lVert q\rVert_{L^{\infty}(\O)}\leq M \mbox{ for some } M>0\}.
 \Eeq 
 We will assume that for all $q \in \Qc_{M}$, $0$ is not an eigenvalue for $\Bc_{q}$ on the domain $D(\Bc_{q})$.  Then given $(f,g)\in H^{7/2}(\PD \O)\times H^{3/2}(\PD \O)$, there is a unique solution to the boundary value problem: 
 \Beq\label{Biharmonic-NavierBVP}
 \Bc_{q} u =0,\quad u|_{\PD \O}=f,\quad \Delta u|_{\PD \O}=g.
 \Eeq
 The boundary conditions are called Navier conditions \cite{PolyharmonicBook} and we define the Dirichlet-to-Neumann map $\Nc_{q}$ for this operator by
\begin{align}
\notag\Nc_{q}: &  H^{7/2}(\PD \O)\times H^{3/2}(\PD \O) \to H^{5/2}(\PD \O)\times H^{1/2}(\PD \O)\\
&\label{D-N Map} (f,g)\to (\frac{\PD u}{\PD \nu}|_{\PD \O}, \frac{\PD (\Delta u)}{\PD \nu}|_{\PD \O}), 
\end{align}
where $u\in H^{4}(\O)$ is the unique solution to \eqref{Biharmonic-NavierBVP}.

We are interested in the inverse problem of determining $q$ from $\Nc_{q}$.  The uniqueness question of determination of $q$ from $\Nc_{q}$ was answered in \cite{Ikehata-BiharmonicPaper, Isakov-BiharmonicPaper} and recently in  \cite{KLU-BiharmonicPaper, KLU-PolyharmonicPaper,Yang-BiharmonicUniquenessPaper} where they showed that unique determination of both zeroth- and first-order perturbations of the birharmonic operator is possible from boundary Neumann data. We note that the papers \cite{KLU-BiharmonicPaper,Yang-BiharmonicUniquenessPaper} also show unique determination of the first order perturbation terms from Neumann data measured on possibly small subsets of the boundary. 

In this paper, we consider the stability question for the determination of $q$ from $\Nc_{q}$ for the operator $\Bc_{q}$. That is, whether one can estimate perturbations of $q$ from perturbations of the Neumann data $\Nc_{q}$.  To the best of the authors' knowledge, stability estimates for inverse problems involving the biharmonic  equation has not been obtained earlier, and the purpose of this paper is to investigate it. We prove a stability estimate of $\ln$-type for the case when the Neumann data is measured on the whole boundary. We then prove a stability estimate of $\ln\ln$-type when the Neumann data is measured on a part of the boundary that is slightly more than half the boundary. \begin{comment}As a consequence of these stability estimates, we prove analogous stability results when Dirichlet conditions are imposed on the boundary instead of the Navier conditions as in \eqref{Biharmonic-NavierBVP} above.
\end{comment}

Our strategy for proving stability estimates follows the methods introduced by Alessandrini in \cite{Alessandrini-StabilityPaper} using complex geometric optics (CGO) solutions where a $\ln$-type stability estimate is proved for the Calder\'on inverse problem \cite{Calderon-InverseProblemPaper}, and by Heck-Wang in \cite{Heck-Wang-StabilityPaper} where a $\ln\ln$-type stability estimate is proved for the Calder\'on inverse problem when the Neumann data is measured on slightly more than half of the boundary. CGO solutions were introduced by Sylvester and Uhlmann in the fundamental paper \cite{Sylvester-Uhlmann-CalderonProblemPaper} to prove global uniqueness for the Calder\'on inverse problem.   The method in Heck and Wang combines CGO solutions and techniques of \cite{Bukhgeim-Uhlmann-InverseProblemPaper} with an analytic continuation result of Vessella \cite{Vessella-ContinuationResult}. Stability estimates for several inverse problems have been obtained in recent years. Apart from the works \cite{Alessandrini-StabilityPaper,Heck-Wang-StabilityPaper} already mentioned, we refer the reader to \cite{Tzou-StabilityPaper,Heck-StabilityPaper,Caro-Pohjola-StabilityPaper, CDR-StabilityPaper} for stability estimates involving  the Calder\'on inverse problem and inverse problems involving the Schr\"odinger or magnetic Schr\"odinger equation.

\section{Statements of the main results}

We now state the main results of this paper. We first consider stability estimates for full boundary measurements and then prove stability estimates when only partial boundary measurements are available.
\subsection{Results for full boundary measurements}
Consider the following norm on $H^{\A}(\PD \O)\times H^{\B}(\PD \O)$ (for simplicity we will denote this space by $H^{\A,\B}(\PD \O))$: 
\Beq\label{Product Space Norm}
\lVert (f,g)\rVert_{H^{\A,\B}(\PD \O)}=\lVert f\rVert_{H^{\A}(\PD \O)}+ \lVert g\rVert_{H^{\B}(\PD \O)} \mbox{ for } (f,g)\in H^{\A,\B}(\PD \O).
\Eeq
Define: 
\[
\lVert \Nc_{q}\rVert= \sup\{ \lVert\Nc_{q}(f,g)\rVert_{H^{\frac{5}{2},\frac{1}{2}}(\PD \O)}: \lVert(f,g)\rVert_{H^{\frac{7}{2},\frac{3}{2}}(\PD \O)}=1\} 
\]
where $\Nc_{q} (f,g)$ is defined in \eqref{D-N Map}.
\begin{theorem}\label{Full-Main Result}
Let $\O\subset \Rb^{n}, n\geq 3$ be a bounded domain with smooth boundary. Consider Equation \eqref{Biharmonic-NavierBVP} for two potentials $q_{1}, q_{2} \in \Qc$.  
Let $\Nc_{q_{1}}$ and $\Nc_{q_{2}}$ be the corresponding Dirichlet-to-Neumann maps measured on $\PD \O$. Then there exists a constant $C=C(\O, n, M)$ such that 
\[
\lVert q_{1}-q_{2}\rVert_{H^{-1}(\O)}^{2}\leq C\lb\lVert \Nc_{q_{1}}-\Nc_{q_{2}}\rVert + \vert \ln \lVert \Nc_{q_{1}}-\Nc_{q_{2}}\rVert\vert^{-\frac{4}{n+2}}\rb.
\]

\end{theorem}
\subsection{Results for partial boundary measurements}
Now we consider the problem of  estimating perturbations of $q$, when the Neumann data $\Nc_{q}$ is measured on a subset of $\PD \O$ that is slightly more than half of the boundary. 
\begin{comment}
As with the case of full boundary measurements, we state the main result in this section for Navier boundary conditions and then the results for Dirichlet boundary conditions follow as a consequence. 
\end{comment}

Before stating the result, we introduce the following notation. Let $\A\in \Sb^{n-1}$ be a unit vector and $\epsilon>0$ be given. Let $\nu(x)$ denote the outer unit normal at $x\in \PD \O$. We define
\begin{align}
\label{Boundary front and back faces positive epsilon}&\PD \O_{+,\ve}=\{x\in \PD \O,\A \cdot \nu(x)>\ve\},\quad  \PD \O_{-,\ve}=\PD\O\setminus \overline{\PD \O_{+,\ve}},\\
\label{Boundary front and back faces}&\PD \O_{+}=\{x\in \PD \O,\A \cdot \nu(x)>0\}, \quad \PD \O_{-}=\PD \O \setminus \overline{\PD \O_{+}}
\end{align}

Now the partial Dirichlet-to-Neumann map is defined as 
\begin{align*}
\wt{\Nc}_{q}: & H^{\frac{7}{2},\frac{3}{2}}(\PD \O) \to H^{\frac{5}{2},\frac{1}{2}}(\PD\O_{-,\ve})\\
& (f,g)\to (\PD_{\nu} u|_{\PD \O_{-,\ve}}, \PD_{\nu}(\Delta u)|_{\PD \O_{-,\ve}}), 
\end{align*}

where $u\in H^{4}(\O)$ is the unique solution to \eqref{Biharmonic-NavierBVP}. As before, we define the norm of $\wt{\Nc}_{q}$ as 
\[
\lVert \wt{\Nc}_{q}\rVert= \sup\{ \lVert\wt{\Nc}_{q}(f,g)\rVert_{H^{\frac{5}{2},\frac{1}{2}}(\PD \O_{-,\ve})}: \lVert(f,g)\rVert_{H^{\frac{7}{2},\frac{3}{2}}(\PD \O)}=1\} 
\]

We have the following stability estimate with partial boundary measurements. 
\begin{theorem}\label{Thm: Partial data result}
Let $\O\subset \Rb^{n}, n\geq 3$ be a bounded domain with smooth boundary. Consider Equation \eqref{Biharmonic-NavierBVP} for two potentials $q_{1}, q_{2} \in \Qc$.  
Let $\wt{\Nc}_{q_{1}}$ and $\wt{\Nc}_{q_{2}}$ be the corresponding Dirichlet-to-Neumann maps measured on $\PD \O_{-,\ve}$. Then there exist constants $C=C(\O, n, M,\ve)$, $K$ and $\theta >0$ such that
\[
\lVert q_{1}-q_{2}\rVert_{H^{-1}(\O)}\leq\Bigg{\{}\lVert \wt{\Nc}_{q_{1}}-\wt{\Nc}_{q_{2}} \rVert+\lb\frac{1}{K} \ln \vert \ln \lVert \wt{\Nc}_{q_{1}}-\wt{\Nc}_{q_{2}}\rVert\vert\rb^{-\frac{2}{\theta}}\Bigg{\}}^{\frac{\theta}{2}}.
\]
\end{theorem}
\section{Preliminary results}
We use the following result from \cite{KLU-BiharmonicPaper,KLU-PolyharmonicPaper}.
\begin{proposition}\cite[Prop. 2.2]{KLU-BiharmonicPaper} (Interior Carleman estimates)\label{Interior Carleman} Let $q\in \Qc_{M}$ and $\vp=x\cdot \A$, $|\A|=1$. There exists an $0<h_{0}=h_{0}(n,M)\ll 1$ and $C=C(n,M)>0$, where $n$ is the dimension and $M$ is the constant in \eqref{potential-space} such that for all $0<h\leq h_{0}\ll 1$ and
$u \in C_{c}^{\infty}(\O)$, we have the following interior estimate:
\[
\lVert e^{\vp/h}h^{4} \Bc_{q}e^{-\vp/h} u\rVert_{L^{2}(\O)}\geq \frac{h^{2}}{C}\lVert u\rVert_{H^{4}_{\mathrm{scl}}(\O)}.
\]
\end{proposition}
This result is based on a Carleman estimate proven in \cite{Salo-Tzou-DiracPaper}. 
\begin{proposition}\cite[Prop. 3.2]{KLU-BiharmonicPaper} (Boundary Carleman estimates)\label{Boundary Carleman}
Let $q\in \Qc_{M}$ and $\vp=x\cdot \A$, $|\A|=1$. Let $\PD \O_{\pm}$ be as in \eqref{Boundary front and back faces}. There exists an $0<h_{0}=h_{0}(n,M)\ll 1$ and $C=C(n,M)>0$, where $n$ is the dimension and $M$ is the constant in \eqref{potential-space} such that for all $0<h\leq h_{0}\ll 1$ and
$u \in D(\Bc_{q})$, we have the following estimate involving boundary terms: 
\begin{align}
\notag &\lVert e^{-\vp/h} h^{4} \Bc_{q}u\rVert_{L^{2}(\O)}+ h^{3/2}\lVert \sqrt{-\A\cdot \nu}e^{-\vp/h}\PD_{\nu}(-h^{2}\triangle u)\rVert_{L^{2}(\PD \O_{-})}\\ &+h^{5/2} \lVert \sqrt{-\A\cdot \nu} e^{-\vp/h} \PD_{\nu} u\rVert_{L^{2}(\PD\O_{-})}\geq \frac{1}{C}\Big{(} h^{2} \lVert e^{-\vp/h} u\rVert_{H^{1}_{\mathrm{scl}}(\O)} \\
\notag &+h^{3/2}\lVert \sqrt{\A\cdot \nu} e^{-\vp/h}\PD_{\nu}(-h^{2}\triangle u)\rVert_{L^{2}(\PD \O_{+})}+h^{5/2} \lVert \sqrt{\A\cdot \nu} e^{-\vp/h} \PD_{\nu} u\rVert_{L^{2}(\PD \O_{+})}\Big{)}.
\end{align} 
\end{proposition}
\begin{comment}
Letting $u\in C_{c}^{\infty}(\O)$ in the previous proposition, we arrive at the following estimate: 
\[\lVert e^{\vp/h} u\rVert_{H^{1}_{\mathrm{scl}}(\O)}\leq Ch^{2}\lVert e^{\vp/h} \Bc_{q}u\rVert_{L^{2}(\O)}.
\]
\end{comment}
Using estimate of Proposition \ref{Interior Carleman}, the following result is proven in \cite{KLU-BiharmonicPaper, KLU-PolyharmonicPaper} which we will require in what follows.
\begin{proposition}\cite[Prop. 2.4]{KLU-BiharmonicPaper}\label{CGO solutions}There exists an $h_{0}=h_{0}(n,M)>0$ and $C=C(n,M)>0$, where $n$ is the dimension and $M$ is the constant in \eqref{potential-space} such that for all $0<h\leq h_{0}\ll 1$, there exist solutions $u(x,\zeta;h)\in H^{4}(\O)$ to $\Bc_{q}u=0$ in $\O$ of the form 
\[
u(x,\zeta;h)=e^{\frac{\I x\cdot \zeta}{h}} (1+ hr(x,\zeta;h)),
\]
with $\zeta\in \Cb^{n}$ satisfying $\zeta\cdot \zeta=0, |\mbox{Re}(\zeta)|=|\mbox{Im}(\zeta)|=1$ and $\lVert r\rVert_{H^{4}_{\mathrm{scl}(\O)}}\leq Ch^{2}$. %The estimates on $h$ and $r$ depend only on $M$ of \eqref{potential-space} and the dimension $n$. 
\end{proposition}
We note that the estimates on $h_{0}$ and $r$ are independent of the potential $q\in \Qc_{M}$.

For proving stability estimates with partial data, we require the following result due to Vessella \cite{Vessella-ContinuationResult}.
\begin{theorem}\cite[Theorem 1]{Vessella-ContinuationResult} \label{Vessella Result} Let $\O\subset \Rb^{n}$ be a bounded open connected set such that for a positive number $r_{0}$ the set $\O_{r}=\{x\in \O: d(x,\PD \O)>r\}$ is connected for every $r\in[0,r_{0}]$. Let $E\subset \O$ be an open set such that $d(E,\PD \O) \geq d_{0}>0$. Let $f$ be an analytic function on $\O$ with the property that 
\[
|D^{\A} f(x)|\leq \frac{C\A!}{\lambda^{|\A|}} \mbox{ for } x\in \O, \A\in (\Nb\cup \{0\})^{n},
\]
where $\lambda, C$ are positive numbers. Then 
\[
|f(x)|\leq (2C)^{1-\g_{1}(|E|/|\O|)}\lb \sup\limits_{E} |f(x)|\rb^{\g_{1}(|E|/|\O|)},
\]
where $|E|$ and $|\O|$ denote the Lebesgue measure of $E$ and $\O$ respectively, $\g_{1}\in (0,1)$ and $\g_{1}$ depends only on $d_{0}$, $\mathrm{diam}(\O), n, r_{0}, \lambda$ and $d(x,\PD \O)$.
\end{theorem}

We also require the following Green formula: 
\begin{align}\label{Biharmonic Greens Identity} 
\int\limits_{\O} (\Bc_{q}u)\overline{v} \D x -\int\limits_{\O} u\overline{\Bc_{q}^{*} v} \D x & =\int\limits_{\PD \O} \PD_{\nu}(\triangle u)\overline{v}\: \D S+\int\limits_{\PD \O} \PD_{\nu} u (\overline{\triangle v})\: \D S\\
\notag & -\int\limits_{\PD \O} (\triangle u)\overline{\PD_{\nu} v}\:\D S -\int\limits_{\PD \O} u(\overline{\PD_{\nu}(\triangle v)})\: \D S.
\end{align}

\begin{comment}
In this section, we find solutions to $\Bc_{q}u=0$ of the form 
\[
u(x,\zeta,h)= e^{\frac{\I x\cdot \zeta}{h}}\lb 1+h r_{1}(x,\zeta;h)\rb.
\]
\end{comment}
\section{Stability estimates with full boundary measurements}
In this section, we prove Theorem \ref{Full-Main Result}.

\begin{proof}[Proof of Theorem \ref{Full-Main Result}]
%Throughout the proof, we will assume $C$ to be a generic constant. 
We start with the Green formula \eqref{Biharmonic Greens Identity} and let $q=q_{1}$ and $u=u_{1}-u_{2}$ and $v\in H^{4}(\O)$ is such that $\Bc_{q_{1}}^{*}v=0$ in $\O$. Here $u_{1}$ and $u_{2}$ are solutions to \eqref{Biharmonic-NavierBVP} for $q$ replaced by $q_{1}$ and $q_{2}$. Then we have 
\Beq \label{Integral equation}
\int\limits_{\O} \lb q_{2}-q_{1}\rb u_{2}\overline{v}\D x=\int\limits_{\PD \O}\PD_{\nu}(\triangle (u_{1}-u_{2}))\overline{v}\: \D S+\int\limits_{\PD \O} \PD_{\nu} (u_{1}-u_{2}) (\overline{\triangle v})\: \D S.
\Eeq
Using Proposition \ref{CGO solutions}, we have solutions to $\Bc_{q_{2}} u_{2}=0$ and $\Bc_{q_{1}}^{*}v=0$ of the form 

\begin{align}
&v(x,\zeta_{1};h)=e^{\frac{\I x\cdot \zeta_{1}}{h}} (1+ hr_{1}(x,\zeta_{1};h)),\\
&u_{2}(x,\zeta_{2};h)=e^{\frac{\I x\cdot \zeta_{2}}{h}} (1+ hr_{2}(x,\zeta_{2};h)),
\end{align}

where
\begin{equation}
\begin{aligned}
&\zeta_{1}= \frac{h \xi}{2}+\sqrt{1-h^2 \frac{\vert \xi \vert ^2}{4}}\beta +\I\alpha,\\
&\zeta_{2}=-\frac{h \xi}{2} +\sqrt{1-h^2 \frac{\vert \xi \vert ^2}{4}}\beta -\I \alpha.
\end{aligned}
\notag
\end{equation}
with $\A$ and $\B$ are unit vectors in $\Rb^{n}$ with $\A, \B$ and $\xi$ are mutually perpendicular vectors and $h$ is such that $h \leq h_{0}$ and $1-h^{2}\frac{\vert \xi \vert^{2}}{4} $ is positive. Substituting $u_{2}$ and $v$ into the left hand side of \eqref{Integral equation}, we get,
\Beq \label{Full data LHS}
\int\limits_{\O} (q_{2}-q_{1})u_{2}\overline{v} \D x= \wh{(q_{2}-q_{1})}(\xi)+  \int_\Omega (q_{2}-q_{1})\ e^{-\I x\cdot \xi} (h \overline{r_{1}} + h r_{2}+h^2 \overline{r_{1}} r_{2})\: \D x.
\Eeq
Calling the second term on the right hand side of the above equation as $I$, we have the following estimate.
\begin{align}\label{Full data estimate for I}
\notag|I|&\leq \int_{\Omega} \vert q_{2}-q_{1} \vert 
(h \vert \bar r_{1} \vert +h \vert r_{2} \vert +h^2 \rvert \bar r_{1} \vert \vert r_{2} \rvert)\: \D x\\
& \leq C (h \Vert r_{1} \Vert_{L^2(\Omega)}+ h \Vert r_{2} \Vert_{L^2(\Omega)} +h^2 \Vert r_{1} \Vert_{L^2(\Omega)} \Vert r_{2} \Vert_{L^2(\Omega)}) \\
%& \leq C (h^3 + h^3 + h^2. h^4)\\
\notag& \leq Ch \mbox{ since } h\ll 1.
\end{align}
\begin{comment}
\[
\begin{aligned}
\lvert\int_\Omega (q_{2}-q_{1})\ e^{-\I x\cdot \xi} (h r_{1} + h \overline{r_{2}}+h^2 r_{1} \overline{r_{2}})\: \D x \rvert \ &\leq \int_{\Omega} \vert q_{2}-q_{1} \vert 
(h \vert \bar r_{1} \vert +h \vert r_{2} \vert +h^2 \rvert \bar r_{1} \vert \vert r_{2} \rvert)\: \D x\\
& \leq C (h \Vert r_{1} \Vert_{L^2(\Omega)}+ h \Vert r_{2} \Vert_{L^2(\Omega)} +h^2 \Vert r_{1} \Vert_{L^2(\Omega)} \Vert r_{2} \Vert_{L^2(\Omega)}) \\
%& \leq C (h^3 + h^3 + h^2. h^4)\\
& \leq Ch, 
\end{aligned}
\]
since $h$ is much smaller compared to $1$.
\end{comment}
Now consider the right hand side of \eqref{Integral equation}. We have
\begin{align*}
&\lvert \int_{\partial \Omega} \partial_{\nu} (\Delta (u_{1}-u_{2}))\bar v \D S +\int_{\partial \Omega} \partial_{\nu}(u_{1}-u_{2})\overline{(\Delta v)} \D S \rvert \\
& \leq \int_{\partial \Omega} \lvert \partial_{\nu} (\Delta (u_{1}-u_{2})) \bar v \rvert \ \D S+ \int_{\partial \Omega} \lvert \partial_{\nu}(u_{1}-u_{2}) (\overline{\Delta v}) \rvert \D S \\
& \leq \Vert \partial_{\nu}(\Delta(u_{1}-u_{2})) \Vert_{L^2(\partial \Omega)} \Vert v \Vert_{L^2(\partial \Omega)} + \Vert \partial_{\nu} (u_{1}-u_{2}) \Vert_{L^2(\partial \Omega)} \Vert \Delta v \Vert_{L^2(\partial \Omega)} \\
& \leq C \big{(} \Vert \partial_{\nu}(\Delta(u_{1}-u_{2})) \Vert_{L^2(\partial \Omega)} \Vert v \Vert_{H^{1}(\Omega)}+ \Vert \partial_{\nu} (u_{1}-u_{2}) \Vert_{L^2(\partial \Omega)} \Vert \Delta v \Vert_{H^{1}(\Omega)}\big{)} \\
&\leq C\big{(}\Vert \partial_{\nu}(\Delta(u_{1}-u_{2})) \Vert_{L^2(\partial \Omega)} + \Vert \partial_{\nu} (u_{1}-u_{2}) \Vert_{L^2(\partial \Omega)}\big{)} \big{(}\Vert v \Vert_{H^{1}(\Omega)} +\Vert \Delta v \Vert_{H^{1}(\Omega)}\big{)} \\
%&\leq C\big{(}\Vert \partial_{\nu}(\Delta(u_{1}-u_{2})) \Vert_{H^\frac{1}{2}(\partial %%\Omega)} + \lVert \partial_{\nu} (u_{1}-u_{2}) \Vert_{H^\frac{5}{2}(\partial \Omega)}%\big{)} \big{(}\lVert v\rVert_{H^{1}(\Omega)} +\lVert \Delta v \rVert_{H^{1}%%           (\Omega)}\big{)}\\
&\leq C\big{(} \lVert \PD_{\nu}(\Delta(u_{1}-u_{2})),\PD_{\nu}(u_{1}-u_{2})\rVert_{H^{\frac{1}{2},\frac{5}{2}}(\PD \O)}\big{)}\big{(}\lVert v\rVert_{H^{1}(\O)}+\lVert \Delta v \rVert_{H^{1}(\Omega)}\big{)}.
\end{align*}
which again is
\begin{align*}
& = C \Vert (\mathcal{N}_{q_{1}}-\mathcal{N}_{q_{2}})(f,g)\Vert_{H^{\frac{5}{2},\frac{1}{2}}(\PD \O)} \big{(}\Vert v \Vert_{H^{1}(\Omega)} +\Vert \Delta v \Vert_{H^{1}(\Omega)}\big{)} \\
&\leq C \Vert\mathcal{N}_{q_{1}}-\mathcal{N}_{q_{2}} \Vert \Vert (f,g)\Vert_{H^{\frac{7}{2},\frac{3}{2}}(\PD \O)} (\Vert v \Vert_{H^{1}(\Omega)} +\Vert \Delta v \Vert_{H^{1}(\Omega)}) \\
&\leq C \lVert\mathcal{N}_{q_{1}}-\mathcal{N}_{q_{2}} \rVert\big{(}\rVert u_{2}\rVert_{H^{4}(\O)}+\lVert \Delta u_{2}\rVert_{H^{2}(\O)}\big{)} \big{(}\Vert v \Vert_{H^{1}(\Omega)} +\Vert \Delta v \Vert_{H^{1}(\Omega)}\big{)}.
\end{align*}
\begin{comment}
\begin{align*}
%&=C \Vert\mathcal{N}_{q_{1}}-\mathcal{N}_{q_{2}} \Vert_{*} (\Vert u_{2} %\Vert_{H^4(\Omega)}+ \Vert \Delta u_{2} \Vert_{H^{2}(\Omega)}) (\Vert v %\Vert_{H^{1}(\Omega)} +\Vert \Delta v \Vert_{H^{1}(\Omega)})
\end{align*}
\end{comment}
We have the following estimates for $\lVert v\rVert_{H^{1}(\O)}$ and $\lVert \Delta v\rVert_{H^{1}(\O)}$. In these estimates, we use that $\O\subset B(0,R)$ for $R>0$ fixed. Then $\lvert e^{\frac{\I x\cdot \zeta_{j}}{h}}\rvert\leq e^{\frac{2R}{h}}$, since $\lvert \zeta_{j}\rvert=2$ for $j=1,2$.
\begin{equation}
\begin{aligned}
\Vert v \Vert_{H^{1}(\Omega)} &\leq \lVert e^\frac{\I x\cdot \zeta_{1}}{h}(1+h r_{1})\rVert_{L^{2}(\Omega)}\NT\NT+\NT\NT\sum\limits^{n}_{k=1} \lVert h e^\frac{\I x\cdot\zeta_{1}}{h} \partial_{x_{k}}r_{1}\NT+\NT\frac{\I}{h}\zeta_{1k} e^\frac{\I x\cdot\zeta_{1}}{h}(1+hr_{1}) \rVert_{L^{2}(\O)} \\
& \leq  e^\frac{2R}{h}\lb 1 +h\Vert r_{1} \Vert_{H^{4}_{\mathrm{scl}}(\Omega)}\rb +\sum\limits_{k=1}^{n} \lb 2 e^\frac{2R}{h}\lb \Vert r_{1} \Vert_{H^{4}_{\mathrm{scl}}(\Omega)}+\frac{1}{h}\rb\rb \\
& \leq C e^\frac{2R}{h}(1+h^2) + \frac{C}{h} e^\frac{2R}{h}(1+h^2)  \leq \frac{C}{h} e^\frac{2R}{h}.
\end{aligned}
\notag
\end{equation}
From straightforward computations, we have the following: 
\begin{align*}
&\Delta v = h e^\frac{\I x\cdot \zeta_{1}}{h} \Delta r_{1} + 2\I e^\frac{\I x\cdot\zeta_{1}}{h} (\zeta_{1}\cdot\nabla r_{1})\\
&\partial _{x_{j}} (\Delta v)= h e^\frac{i x\cdot\zeta_{1}}{h} \partial_{x_{j}}(\Delta r_{1})+ i\zeta_{1j}e^\frac{\I x\cdot\zeta_{1}}{h}\Delta r_{1}+ 2\I e^\frac{\I x\cdot\zeta_{1}}{h} \partial_{x_{j}}(\zeta_{2}\cdot\nabla r_{1})\\
&\hspace{0.75in}-\frac{2}{h}\zeta_{1j} e^\frac{\I x\cdot\zeta_{1}}{h}(\zeta_{1}\cdot\nabla r_{1}).
\end{align*}
\begin{equation}
\begin{aligned}
\partial_{x_{k}} \partial_{x_{j}} (\Delta u_{2}) &=h e^\frac{\I x\cdot\zeta_{2}}{h} \partial_{x_{k}} \partial_{x_{j}} (\Delta r_{2})+ \I \zeta_{2k} e^\frac{\I x\cdot\zeta_{2}}{h} \partial_{x_{j}}(\Delta r_{2})+ \I \zeta_{2j} e^\frac{\I x\cdot\zeta_{2}}{h}\partial_{x_{k}}(\Delta r_{2}) \\
& -\frac{1}{h} \zeta_{2j} \zeta_{2k} e^\frac{\I x\cdot\zeta_{2}}{h} \Delta r_{2}+ 2\I e^\frac{\I x\cdot\zeta_{2}}{h}\partial_{x_{k}} \partial_{x_{j}}(\zeta_{2}\cdot\nabla r_{2})\\
&-\frac{2}{h} \zeta_{2k} e^\frac{\I x\cdot\zeta_{2}}{h} \partial_{x_{j}}(\zeta_{2}\cdot\nabla r_{2}) 
-\frac{2}{h} \zeta_{2j}e^\frac{\I x\cdot\zeta_{2}}{h} \partial_{x_{k}}(\zeta_{2}\cdot\nabla r_{2}) \\
&-\frac{2\I}{h^2} \zeta_{2j} \zeta_{2k} e^\frac{\I x\cdot\zeta_{2}}{h} (\zeta_{2}\cdot\nabla r_{2}).
\end{aligned}
\notag
\end{equation}
\begin{equation}
\begin{aligned}
\partial_{x_{m}} \partial_{x_{l}} \partial_{x_{k}} \partial_{x_{j}} u_{2} &= h e^\frac{\I x\cdot\zeta_{2}}{h} \partial_{x_{m}} \partial_{x_{l}} \partial_{x_{k}} \partial_{x_{j}} r_{2} + \I\zeta_{2m} e^\frac{\I x\cdot\zeta_{2}}{h}\partial_{x_{l}} \partial_{x_{k}} \partial_{x_{j}} r_{2} \\
& +\I \zeta_{2l} e^\frac{\I x\cdot\zeta_{2}}{h} \partial_{x_{m}} \partial_{x_{k}} \partial_{x_{j}} r_{2} 
 -\frac{1}{h} \zeta_{2l} \zeta_{2m} e^\frac{\I x\cdot\zeta_{2}}{h} \partial_{x_{k}} \partial_{x_{j}} r_{2} \\
 &-\frac{1}{h}\zeta_{2m} \zeta_{2k} e^\frac{\I x\cdot\zeta_{2}}{h} \partial_{x_{l}} \partial_{x_{j}} r_{2} 
 + \I \zeta_{2k} e^\frac{\I x\cdot\zeta_{2}}{h} \partial_{x_{m}} \partial_{x_{l}} \partial_{x_{j}} r_{2} \\
 &
 -\frac{1}{h} \zeta_{2l} \zeta_{2k} e^\frac{\I x\cdot\zeta_{2}}{h} \partial_{x_{m}} \partial_{x_{j}} r_{2} -\frac{\I}{h^2} \zeta_{2m} \zeta_{2l} \zeta_{2k} e^\frac{\I x\cdot\zeta_{2}}{h} \partial_{x_{j}} r_{2}\\
 &
 -\frac{1}{h} \zeta_{2m}\zeta_{2j} e^\frac{\I x\cdot\zeta_{2}}{h} \partial_{x_{l}} \partial_{x_{k}} r_{2} 
 +\I\zeta_{2j} e^\frac{\I x\cdot\zeta_{2}}{h} \partial_{x_{m}} \partial_{x_{l}} \partial_{x_{k}} r_{2} \\
 &-\frac{1}{h} \zeta_{2l} \zeta_{2j} e^\frac{\I x\cdot\zeta_{2}}{h} \partial_{x_{m}} \partial_{x_{k}} r_{2} -\frac{\I}{h^2} \zeta_{2m} \zeta_{2l} \zeta_{2j} e^\frac{\I x\cdot\zeta_{2}}{h} \partial_{x_{k}}r_{2} \\
 &-\frac{\I}{h^2} \zeta_{2m} \zeta_{2j} \zeta_{2k} e^\frac{\I x\cdot \zeta_{2}}{h} \partial_{x_{l}} r_{2} 
 -\frac{1}{h} \zeta_{2j} \zeta_{2k} e^\frac{\I x\cdot\zeta_{2}}{h} \partial_{x_{m}} \partial_{x_{l}} r_{2} \\
 &-\frac{\I}{h^2} \zeta_{2l} \zeta_{2j} \zeta_{2k} e^\frac{\I x\cdot\zeta_{2}}{h} \partial_{x_{m}}r_{2} 
 +\frac{1}{h^4} \zeta_{2m} \zeta_{2l} \zeta_{2j} \zeta_{2k} e^\frac{\I x\cdot\zeta_{2}}{h} (1+h r_{2}).
\end{aligned}
\notag
\end{equation}
Now using the above derivatives, it is straightforward to show the following:
\begin{align}
\label{H1 estimate of CGO}&\lVert \Delta v\rVert_{H^{1}(\O)}\leq \frac{Ce^{\frac{2R}{h}}}{h^{2}} \lVert r_{1}\rVert_{H^{4}_{\mathrm{scl}(\O)}}\leq \frac{C}{h}e^{\frac{2R}{h}}.\\
\label{H2 estimate of CGO}&\lVert \Delta u_{2} \rVert_{H^2(\Omega)} \leq \frac{C}{h}e^\frac{2R}{h} .\\
\label{H4 estimate of CGO}&\Vert u_{2} \Vert_{H^4(\Omega)} \leq \frac{C}{h^4}e^\frac{2R}{h} .
\end{align}
Therefore we have
\begin{equation}
\begin{aligned}
\rvert \int_{\partial \Omega} \partial_{\nu} (\Delta (u_{1}-u_{2}))\bar v \D S &+\int_{\partial \Omega} \partial_{\nu}(u_{1}-u_{2})\overline{(\Delta v)} \D S \lvert \\
& \leq C \Vert \mathcal{N}_{q_{1}}-\mathcal{N}_{q_{2}} \Vert(\frac{C}{h^4}e^\frac{2R}{h}+\frac{C}{h}e^\frac{2R}{h})(\frac{C}{h}e^\frac{2R}{h}+\frac{C}{h} e^\frac{2R}{h})\\
& \leq  \frac{C}{h^4} e^\frac{2R}{h}. \frac{C}{h} e^\frac{2R}{h} \Vert \mathcal{N}_{q_{1}}-\mathcal{N}_{q_{2}} \Vert\leq \frac{C}{h^5} e^\frac{4R}{h} \Vert \mathcal{N}_{q_{1}}-\mathcal{N}_{q_{2}} \Vert.
\end{aligned}
\notag
\end{equation}
Now using the fact that $\frac{1}{h} \leq e^\frac{R}{h}$, we obtain

\[
\vert \int_{\partial \Omega} \partial_{\nu} (\Delta (u_{1}-u_{2}))\bar v \D S +\int_{\partial \Omega} \partial_{\nu}(u_{1}-u_{2})\overline{(\Delta v)} \D S \vert \leq C e^\frac{9R}{h} \Vert \mathcal{N}_{q_{1}}-\mathcal{N}_{q_{2}} \Vert .
\]
Extending $q_{1},q_{2}$ to $\mathbb{R}^{n}$ by $0$ and using \eqref{Full data LHS} and \eqref{Full data estimate for I}, we get the estimate
\begin{equation}
\vert \wh{(q_{1}-q_{2})} (\xi)\vert \leq C (e^\frac{9R}{h}\Vert \mathcal{N}_{q_{1}}-\mathcal{N}_{q_{2}} \Vert+ h)
\notag
\end{equation}
Now
\begin{equation}
 \begin{aligned}
  \Vert q_{1}-q_{2} \Vert^{2}_{H^{-1}(\Omega)} &\leq \Vert q_{1}-q_{2} \Vert^{2}_{H^{-1}(\mathbb{R}^{n})} \\
   & = \int_{\vert \xi \vert \leq \rho} \frac{\vert \wh{(q_{1}-q_{2})}(\xi)\vert^{2}}{1+\vert \xi \vert^{2}} \D \xi + \int_{\vert \xi \vert > \rho} \frac{\vert \wh{(q_{1}-q_{2})}(\xi)\vert^{2}}{1+\vert \xi \vert^{2}} \D \xi,
 \end{aligned}
\notag
\end{equation}
for appropriate $\rho$  to be chosen later.\\
But
\begin{equation}
 \begin{aligned}
  \int_{\vert \xi \vert > \rho} \frac{\vert \wh{(q_{1}-q_{2})}(\xi)\vert^{2}}{1+\vert \xi \vert^{2}} \ d \xi & \leq 
\int_{\vert \xi \vert > \rho} \frac{\vert \wh{(q_{1}-q_{2})} (\xi) \vert^2}{1+\rho^2} \ d \xi \\
& \leq \frac{1}{\rho^2} \Vert q_{1}-q_{2} \Vert^{2}_{L^2(\mathbb{R}^{n})} \leq \frac{C}{\rho^{2}}
 \end{aligned}
\notag
\end{equation}
and
\begin{equation}
 \begin{aligned}
  \int_{\vert \xi \vert \leq \rho} \frac{\vert \wh{(q_{1}-q_{2})}(\xi)\vert^{2}}{1+\vert \xi \vert^{2}} \D \xi & \leq 
C \int_{\vert \xi \vert \leq \rho} \frac{(e^\frac{9R}{h} \Vert \mathcal{N}_{q_{1}} -\mathcal{N}_{q_{2}} \Vert+h)^{2}}{1+\vert \xi \vert^{2}} \D \xi  \\
   & \leq C (e^\frac{18R}{h} \Vert \mathcal{N}_{q_{1}}-\mathcal{N}_{q_{2}} \Vert^{2}+ h^{2}) \int_{\vert \xi \vert \leq \rho} \frac{d \xi}{1+\vert \xi \vert^{2}} \\
   & \leq C \rho^{n} (e^\frac{18R}{h} \Vert \mathcal{N}_{q_{1}}-\mathcal{N}_{q_{2}} \Vert^{2}+h^{2})
 \end{aligned}
\notag
\end{equation}
Therefore 
\begin{equation}
 \Vert q_{1}-q_{2} \Vert^{2}_{H^{-1}(\Omega)} \leq C \rho^{n} e^\frac{18R}{h} \Vert \mathcal{N}_{q_{1}}-\mathcal{N}_{q_{2}} \Vert^{2}+C \rho^{n} h^{2} + \frac{C}{\rho^{2}}\ .
 \notag
\end{equation}
Now assume that  $\Vert \mathcal{N}_{q_{1}}-\mathcal{N}_{q_{2}} \Vert < \delta = e^-{\frac{20R}{h_{0}}}$. 
Then we choose $\rho=\{\frac{1}{20R} \vert \ln \ \Vert \mathcal{N}_{q_{1}}-\mathcal{N}_{q_{2}} \Vert \vert \}^{\frac{2}{n+2}}.$ 
Further let $h=\frac{1}{\rho^{\frac{n+2}{2}}}$. With this choice of $h$, we show that $h<h_{0}$ and $1-h^{2} |\xi|^{2}/4>0$ for $|\xi|<\rho$.  The fact that $h<h_{0}$ follows from these inequalities: 
 \begin{align*}
  \Vert \mathcal{N}_{q_{1}}-\mathcal{N}_{q_{2}}  \Vert &< e^{-\frac{20R}{h_{0}}}\ll 1 \\
  \Rightarrow \ln \Vert \mathcal{N}_{q_{1}}-\mathcal{N}_{q_{2}}  \Vert &< -\frac{20R}{h_{0}}\\
  \Rightarrow \vert \ln \Vert \mathcal{N}_{q_{1}}-\mathcal{N}_{q_{2}}  \Vert \vert &> \frac{20R}{h_{0}} \\
  \Rightarrow \frac{1}{20R} \vert \ln \Vert \mathcal{N}_{q_{1}}-\mathcal{N}_{q_{2}}  \Vert \vert &> \frac{1}{h_{0}} \\
  %\Rightarrow \frac{1}{h} &> \frac{1}{h_{0}} \\
  \Rightarrow h &< h_{0}.
 \end{align*}
%\notag
%\end{equation}
Now we show that $1-h^{2}\frac{\vert \xi \vert^{2}}{4} >0 $ for $|\xi|<\rho$. 
We have that 
\[\rho^{n}=\{\frac{1}{20R} \vert \ln \Vert  \mathcal{N}_{q_{1}}-\mathcal{N}_{q_{2}} \Vert \vert \}^{\frac{2n}{n+2}}. \] 
Since $\frac{2n}{n+2}>1 $ and  $\frac{1}{20R} \vert \ln \Vert  \mathcal{N}_{q_{1}}-\mathcal{N}_{q_{2}} \Vert \vert >1 $, we have  that $\rho^{n} >1$.\\
Hence
$$h^{2}\frac{\vert \xi \vert^{2}}{4}<h^{2} \frac{\rho^{2}}{4} = \frac{1}{4 \rho^{n}} <1 $$
and so $1-h^{2}\frac{\vert \xi \vert^{2}}{4}>0 .$
\\
\\
Therefore
\begin{equation}
\begin{aligned}
\Vert q_{1}-q_{2} \Vert^{2}_{H^{-1}(\Omega)} &\leq \frac{C}{h^\frac{2n}{n+2}} e^\frac{18R}{h} \Vert \mathcal{N}_{q_{1}}-\mathcal{N}_{q_{2}}  \Vert^{2}
+ C h^\frac{4}{n+2} \\
& \leq C e^\frac{20R}{h} \Vert \mathcal{N}_{q_{1}}-\mathcal{N}_{q_{2}}   \Vert^{2} + C h^\frac{4}{n+2}.
\end{aligned}
 \notag
\end{equation}
and since $\frac{1}{h}=\frac{1}{20R} \vert \ln \ \Vert \mathcal{N}_{q_{1}}-\mathcal{N}_{q_{2}} \Vert\vert $,    
we then  obtain the estimate 
\begin{equation}
\Vert q_{1}-q_{2}  \Vert^{2}_{H^{-1}(\Omega)}\leq C(\Vert \mathcal{N}_{q_{1}}-\mathcal{N}_{q_{2}}  \Vert+ \vert \ln \Vert \mathcal{N}_{q_{1}}-\mathcal{N}_{q_{2}}   \Vert     \vert^{-\frac{4}{n+2}}),
 \notag
\end{equation}
when $\Vert \mathcal{N}_{q_{1}}-\mathcal{N}_{q_{2}} \Vert < \delta = e^-{\frac{20R}{h_{0}}}$. 

The case when $\Vert \mathcal{N}_{q_{1}}-\mathcal{N}_{q_{2}} \Vert \geq \delta$ follows from the continuous inclusions $$L^{\infty}(\Omega) \hookrightarrow 
L^{2}(\Omega) \hookrightarrow H^{-1}(\Omega).$$ In other words, we have
\begin{equation}
\begin{aligned}
 \Vert q_{1}-q_{2} \Vert^{2}_{H^{-1}(\Omega)} \leq C \Vert q_{1}-q_{2} \Vert^{2}_{L^{\infty}(\Omega)} \leq \frac{4CM^{2}}{\delta} \delta
 \leq \frac{4CM^{2}}{\delta} \Vert \mathcal{N}_{q_{1}}-\mathcal{N}_{q_{2}} \Vert
\end{aligned}
\notag
\end{equation}
and hence the desired estimate follows.
This concludes the proof.
\end{proof}
\section{Stability estimate for slightly more than half data}

Here we prove stability estimates for the partial data case. In the appendix, we include a proof of the identifiability in this case using linear Carleman weights. We would be using
a few estimates derived therein in this section.

\begin{proof} [Proof of Theorem \ref{Thm: Partial data result}]
We begin with the following identity as at the beginning of Theorem \ref{Full-Main Result} and rewrite it as 
\begin{align}
\notag\int\limits_{\O} \lb q_{2}-q_{1}\rb u_{2}\overline{v}\D x&=\int\limits_{\PD \O}\PD_{\nu}(\triangle (u_{1}-u_{2}))\overline{v}\: \D S+\int\limits_{\PD \O} \PD_{\nu} (u_{1}-u_{2}) (\overline{\triangle v})\: \D S\\
\label{Terms on more than half boundary} &= \int\limits_{\PD \O_{-,\ve}}\PD_{\nu}(\triangle (u_{1}-u_{2}))\overline{v}\: \D S  +\int\limits_{\PD \O_{-,\ve}} \PD_{\nu} (u_{1}-u_{2}) (\overline{\triangle v})\: \D S\\
\label{Terms on less than half boundary}&+ \int\limits_{\PD \O_{+,\ve}}\PD_{\nu}(\triangle (u_{1}-u_{2}))\overline{v}\: \D S  +\int\limits_{\PD \O_{+,\ve}} \PD_{\nu} (u_{1}-u_{2}) (\overline{\triangle v})\: \D S.
\end{align}
\begin{comment}
The constant $C$ in the proof below would be a generic constant which might as well depend upon $\epsilon $. 
\begin{equation}
\begin{aligned}
\int_{\Omega} (q_{2}-q_{1})u_{2} \bar v \ dx  &=-\int_{\partial \Omega_{+,\epsilon}} \partial_{\nu}(-\Delta(u_{1}-u_{2}))\bar v \ \D S -\int_{\partial \Omega_{+,\epsilon}} \partial_{\nu}(u_{1}-u_{2}) \overline{(-\Delta v)}\ \D S \\
& -\int_{\partial \Omega_{-,\epsilon}} \partial_{\nu}(-\Delta(u_{1}-u_{2}))\bar v \ \D S -\int_{\partial \Omega_{-,\epsilon}} \partial_{\nu}(u_{1}-u_{2}) \overline{(-\Delta v)}\ \D S
\end{aligned}
\notag
\end{equation}
\end{comment}
We estimate the terms in \eqref{Terms on more than half boundary}. Proceeding as with the full data case, we have 
\begin{equation}
\begin{aligned}
& \lvert \int_{\partial \Omega_{-,\epsilon}} \partial_{\nu} (\Delta (u_{1}-u_{2}))\bar v\: \D S +\int_{\partial \Omega_{-,\epsilon}} \partial_{\nu}(u_{1}-u_{2})\overline{(\Delta v)} \ \D S \rvert \\
%& \leq \Vert \partial_{\nu}(\Delta(u_{1}-u_{2})) \Vert_{L^2(\partial %\Omega_{-,\epsilon})} \Vert v \Vert_{L^2(\partial \Omega_{-,\epsilon})} + %\Vert \partial_{\nu} (u_{1}-u_{2}) \Vert_{L^2(\partial \Omega_{-,\epsilon})} %\Vert \Delta v \Vert_{L^2(\partial \Omega_{-,\epsilon})} \\
%& \leq C (\Vert \partial_{\nu}(\Delta(u_{1}-u_{2})) \Vert_{L^2(\partial %\Omega_{-,\epsilon})} \Vert v \Vert_{H^{1}(\Omega)}+ \Vert \partial_{\nu} %(u_{1}-u_{2}) \Vert_{L^2(\partial \Omega_{-,\epsilon})} \Vert \Delta v %\Vert_{H^{1}(\Omega)}) \\
%& \leq C(\Vert \partial_{\nu}(\Delta(u_{1}-u_{2})) \Vert_{L^2(\partial %\Omega_{-,\epsilon})} + \Vert \partial_{\nu} (u_{1}-u_{2}) %\Vert_{L^2(\partial \Omega_{-,\epsilon})}) (\Vert v \Vert_{H^{1}(\Omega)} +%\Vert \Delta v \Vert_{H^{1}(\Omega)}) \\
& \leq C\lb\Vert \partial_{\nu}(\Delta(u_{1}-u_{2})) \Vert_{H^\frac{1}{2}(\partial \Omega_{-,\epsilon})} + \Vert \partial_{\nu} (u_{1}-u_{2}) \Vert_{H^\frac{5}{2}(\partial \Omega_{-,\epsilon})}\rb\\
&\hspace{0.4in} \lb\Vert v \Vert_{H^{1}(\Omega)} +\Vert \Delta v \Vert_{H^{1}(\Omega)}\rb \\
& = C \Vert (\wt {\mathcal{N}}_{q_{1}}-\wt{\mathcal{N}}_{q_{2}})(f,g)\Vert_{H^{\frac{5}{2},\frac{1}{2}}(\partial \Omega_{-,\epsilon})} (\Vert v \Vert_{H^{1}(\Omega)} +\Vert \Delta v \Vert_{H^{1}(\Omega)}) \\
& \leq C \Vert \wt{\mathcal{N}}_{q_{1}}-\wt{\mathcal{N}}_{q_{2}} \Vert\Vert (f,g)\Vert_{H^{\frac{7}{2},\frac{3}{2}}(\partial \Omega)} (\Vert v \Vert_{H^{1}(\Omega)} +\Vert \Delta v \Vert_{H^{1}(\Omega)}) \\
%& \leq C \Vert \wt{\mathcal{N}}_{q_{1}}-\wt{\mathcal{N}}_{q_{2}} \Vert %\Vert (u_{2},\Delta u_{2})\Vert_{H^4(\Omega)\times H^{2}(\Omega)} (\Vert v %\Vert_{H^{1}(\Omega)} +\Vert \Delta v \Vert_{H^{1}(\Omega)}) \\
& \leq C \Vert \wt{\mathcal{N}}_{q_{1}}-\wt{\mathcal{N}}_{q_{2}} \Vert (\Vert u_{2} \Vert_{H^4(\Omega)}+ \Vert \Delta u_{2} \Vert_{H^{2}(\Omega)}) (\Vert v \Vert_{H^{1}(\Omega)} +\Vert \Delta v \Vert_{H^{1}(\Omega)}) \\
& \leq C e^\frac{9R}{h} \Vert \wt{\mathcal{N}}_{q_{1}}-\wt{\mathcal{N}}_{q_{2}} \Vert.
\end{aligned}
\notag
\end{equation}
Now we estimate the terms in \eqref{Terms on less than half boundary}. 

We first have 
%Also proceeding similarly as in the case of identifiability, we conclude that
\begin{equation}
\begin{aligned}
\lvert \int_{\partial \Omega_{+,\epsilon}} \NT\NT\partial_{\nu}(\Delta(u_{1}-u_{2}))\bar v \ \D S \rvert &= \vert \int_{\partial \Omega_{+,\epsilon}} e^{-\frac{x\cdot \A}{h}} \partial_{\nu}(\Delta(u_{1}-u_{2})) e^{\frac{x\cdot \A}{h}} \bar v \ \D S \vert \\
&\leq \Vert e^{-\frac{x\cdot \A}{h}} \partial_{\nu} (\Delta(u_{1}-u_{2})) \Vert_{L^2(\partial \Omega_{+,\epsilon})} \Vert e^{\frac{x\cdot \A}{h}} \bar v \Vert_{L^2(\partial \Omega_{+,\epsilon})} \\
& \leq C \Vert e^{-\frac{x\cdot \A}{h}} \partial_{\nu} (\Delta(u_{1}-u_{2})) \Vert_{L^2(\partial \Omega_{+,\epsilon})}
\end{aligned}
\notag
\end{equation}
and
\begin{equation}
\begin{aligned}
\vert \int_{\partial \Omega_{+,\epsilon}} \NT\NT\partial_{\nu} (u_{1}-u_{2}) \overline{(\Delta v)} \ \D S \vert &= \vert \int_{\partial \Omega_{+,\epsilon}} e^{-\frac{x\cdot \A}{h}} \partial_{\nu} (u_{1}-u_{2}) e^{\frac{x\cdot \A}{h}} \overline{(\Delta v)} \ \D S \vert \\
& \leq \Vert e^{-\frac{x\cdot \A}{h}} \partial_{\nu} (u_{1}-u_{2}) \Vert_{L^2(\partial \Omega_{+,\epsilon})} \Vert e^{\frac{x\cdot \A}{h}} \overline{(\Delta v)} \Vert_{L^2(\partial \Omega_{+,\epsilon})} \\
& \leq C  \Vert e^{-\frac{x\cdot \A}{h}} \partial_{\nu} (u_{1}-u_{2}) \Vert_{L^2(\partial \Omega_{+,\epsilon})}.
\end{aligned}
\notag
\end{equation}
By the boundary Carleman estimate, we have for $\ve>0$,
\begin{equation}
\begin{aligned}
& h^{\frac{3}{2}}\Vert \sqrt{\A\cdot\nu} \ e^{-\frac{x\cdot\alpha}{h}}\ \partial_{\nu}(-h^2 \Delta u) \Vert_{L^2(\partial \Omega_{+})}
\leq C( \Vert e^{-\frac{x\cdot\alpha}{h}} (h^4 \mathcal{\Bc}_{q_{1}})u \Vert_{L^2(\Omega)}+ \\
 & h^\frac{3}{2}  \Vert \sqrt{-\alpha\cdot\nu}\  e^{-\frac{x\cdot\alpha}{h}} \partial_{\nu}(-h^2 \Delta u) \Vert_{L^2(\partial \Omega_{-})}
+ h^\frac{5}{2} \Vert \sqrt{-(\alpha\cdot\nu)} \ e^{-\frac{x\cdot\alpha}{h}} \partial_{\nu}u \Vert_{L^2(\partial \Omega_{-})}) \\
%& = C (h^\frac{5}{2} \Vert e^{-\frac{x\cdot\alpha}{h}} (\mathcal{L}%_{q_{1}})u \Vert_{L^2(\Omega)}+h^2 \Vert \sqrt{-\A\cdot\nu}\  e^{-\frac{x%\cdot\alpha}{h}} \partial_{\nu}(-\Delta u)  \Vert_{L^2(\partial \Omega_{-})}\%\ 
%&\ \ \ \ \ \ \ \ \ \ \ \ \ \ \ \ \ \ \ \ \ \ \ \ \ \ \ \ \ \ \ \ \ \ \ \ \ \ \ \ \ \ +h \Vert \sqrt{-\A\cdot\nu} \ e^{-\frac{x\cdot\alpha}{h}} %\partial_{\nu}u \Vert_{L^2(\partial \Omega_{-})})
\end{aligned}
\notag
\end{equation}
We then have
\begin{align*}
\Vert \sqrt{\alpha\cdot\nu}\ e^{-\frac{x\cdot\alpha}{h}} \partial_{\nu}(\Delta u) \Vert_{L^2(\partial \Omega_{+})} \NT\NT
 \leq & C \Big{(}\NT \sqrt{h} \Vert e^{-\frac{x\cdot\alpha}{h}}\Bc_{q_{1}}u \Vert_{L^2(\Omega)}\\
 &+ \NT \frac{1}{h} \Vert \sqrt{-\alpha\cdot\nu} \ e^{-\frac{x\cdot \A}{h}} \partial_{\nu}u \Vert_{L^2(\partial \Omega_{-})}\\
 & +\Vert \sqrt{-\alpha\cdot\nu}\  e^{-\frac{x\cdot\alpha}{h}} \partial_{\nu}(\Delta u)  \Vert_{L^2(\partial \Omega_{-})}\Big{)}.
\end{align*}
Since on $\PD \O_{+,\ve}$, $\A\cdot \nu>\ve$, we have
\begin{align*}
 \Vert e^{-\frac{x\cdot\alpha}{h}} \partial_{\nu}(\Delta u) \Vert_{L^2(\partial \Omega_{+,\epsilon})} \leq & \frac{C}{\sqrt{\epsilon}} \Big{(}\sqrt{h}\Vert e^{-\frac{x\cdot\alpha}{h}}\Bc_{q_{1}}u \Vert_{L^2(\Omega)} +\frac{1}{h} \Vert \sqrt{-(\A\cdot \nu)} \ e^{-\frac{x\cdot \A}{h}} \partial_{\nu}u \Vert_{L^2(\partial \Omega_{-})}\\
&+\Vert \sqrt{-(\A\cdot \nu)}\  e^{-\frac{x\cdot \A}{h}} \partial_{\nu}(\Delta u)  \Vert_{L^2(\partial \Omega_{-})}\Big{)}. \\
& \leq \frac{C}{\sqrt{\ve}} \Big{(}\sqrt{h}\Vert e^{-\frac{x\cdot \A}{h}}\Bc_{q_{1}}u \Vert_{L^2(\Omega)}  \\
&+\sqrt{-\inf_{\partial \Omega_{-}} (\A\cdot \nu)} \Vert e^{-\frac{x\cdot \A}{h}} \partial_{\nu}(\Delta u) \Vert_{L^2(\partial \Omega_{-,\epsilon})} \\
&  + \sqrt{- \inf_{\partial \Omega_{-}} (\A\cdot \nu)} \Vert e^{-\frac{x\cdot\alpha}{h}} \partial_{\nu} u \Vert_{L^2(\partial \Omega_{-,\epsilon})}\Big{)}.
\end{align*}
Therefore, 
\begin{align*}
  \vert \int_{\partial \Omega_{+,\epsilon}} \partial_{\nu}(\Delta(u_{1}-u_{2})) \bar{v} \ \D S \vert &\leq \frac{C}{\sqrt{\epsilon}} \sqrt{-\displaystyle \inf_{\partial \Omega_{-}} (\A\cdot \nu)} \ \Vert e^{-\frac{x\cdot \A}{h}} \partial_{\nu}(\Delta(u_{1}-u_{2})) \Vert_{L^{2}(\partial \Omega_{-,\epsilon})}\\
  &+C \sqrt{\frac{h}{\epsilon}} \Vert e^{-\frac{x\cdot \A}{h}} (\mathcal{B}_{q_{1}})(u_{1}-u_{2}) \Vert_{L^{2}(\Omega)} 
    \\
   &+\frac{C}{h \sqrt{\epsilon}} \sqrt{-\inf_{\partial \Omega_{-}} (\A\cdot \nu)} \ \Vert e^{-\frac{x\cdot \A}{h}} \partial_{\nu}(u_{1}-u_{2})  \Vert_{L^{2}(\partial \Omega_{-,\epsilon})}.
    \end{align*}
    Now we have 
    \[
    \Vert e^{-\frac{x\cdot \A}{h}} \mathcal{B}_{q_{1}}(u_{1}-u_{2}) \Vert_{L^{2}(\Omega)} =\lVert e^{-\frac{x\cdot \A}{h}} (q_{1}-q_{2})u_{2}\rVert_{L^{2}(\O)}.
    \]
   Using the CGO solutions from Proposition \eqref{CGO solutions},
    \[
    u_{2}(x,\zeta_{2};h)=e^{ \frac{\I x\cdot \zeta_{2}}{h}}(1+hr_{2}(x,\zeta_{2};h)) \mbox{ where } 
    \]
    \[
    \zeta_{2}=-\frac{h \xi}{2} +\sqrt{1-h^2 \frac{\vert \xi \vert ^2}{4}}\beta -\I \alpha,
    \]
    we have 
  
    \[
    \lVert e^{-\frac{x\cdot \A}{h}} (q_{1}-q_{2})u_{2}\rVert_{L^{2}(\O)}\leq C.
    \]
    Therefore 
    \begin{align*}
    \lvert \int_{\partial \Omega_{+,\epsilon}} \NT\NT\NT\NT\NT\NT\partial_{\nu}(\Delta(u_{1}-u_{2})) \bar{v} \ \D S \rvert&\leq C \sqrt{h}+C e^{\frac{R}{h}} \Vert \partial_{\nu}(\Delta(u_{1}-u_{2})) \Vert_{L^{2}(\partial \Omega_{-,\epsilon})}\\
    &+\frac{C}{h} e^{\frac{R}{h}} \Vert \partial_{\nu}(u_{1}-u_{2}) \Vert_{L^{2}(\partial \Omega_{-,\epsilon})}.\\
  & \leq C \Big{(}\sqrt{h}+\frac{e^{\frac{R}{h}}}{h}\Big{(}\Vert \partial_{\nu}(\Delta(u_{1}-u_{2})) \Vert_{L^{2}(\partial \Omega_{-,\epsilon})}\\
  &+\Vert \partial_{\nu}(u_{1}-u_{2}) \Vert_{L^{2}(\partial \Omega_{-,\epsilon})}\Big{)}\Big{)}. \\
  & \leq C \sqrt{h}\NT+\NT\frac{C}{h}e^{\frac{R}{h}} \Vert \wt{\mathcal{N}}_{q_{1}} - \wt{\mathcal{N}}_{q_{2}} \Vert (\Vert u_{2} \Vert_{H^{4}(\Omega)}\NT+\NT\Vert \Delta u_{2} \Vert_{H^{2}(\Omega)}) \\
  \intertext{and using the estimates \eqref{H4 estimate of CGO} and \eqref{H2 estimate of CGO}, it follows that this is}
  & \leq C \sqrt{h}+\frac{C}{h^{5}}e^{\frac{3R}{h}} \Vert \wt{\mathcal{N}}_{q_{1}}-\wt{\mathcal{N}}_{q_{2}} \Vert\\
  &\leq C \sqrt{h}+C e^{\frac{8R}{h}} \Vert \wt{\mathcal{N}}_{q_{1}}-\wt{\mathcal{N}}_{q_{2}} \Vert,
 \end{align*}
 where the constant $C$ now depends upon $\ve$.\\
From the boundary Carleman estimate, we also have
\begin{equation}
\begin{aligned} 
h^{\frac{5}{2}}\Vert \sqrt{\A\cdot \nu}\ e^{-\frac{x\cdot \A}{h}} \partial_{\nu}u \Vert_{L^2(\partial \Omega_{+,\ve})} \leq & C (\Vert    e^{-\frac{x\cdot \A}{h}} (h^4 \mathcal{B}_{q_{1}})u \Vert_{L^2(\Omega)}\\
&+ h^\frac{3}{2}  \Vert \sqrt{-\A\cdot \nu}\  e^{-\frac{x\cdot \A}{h}} \partial_{\nu}(-h^2 \Delta u) \Vert_{L^2(\partial \Omega_{-},\ve)} \\ 
& + h^\frac{5}{2} \Vert \sqrt{-\A\cdot \nu} \ e^{-\frac{x\cdot \A}{h}} \partial_{\nu}u \Vert_{L^2(\partial \Omega_{-},\ve)}) \\
%& = C\Big{(}h^{\frac{3}{2}} \Vert e^{-\frac{x\cdot \A}{h}} \mathcal{B}_{q_{1}}u  \Vert_{L^{2}(\Omega)}\\&+ h \Vert \sqrt{-\A\cdot \nu}\  e^{-\frac{x\cdot \A}{h}} \partial_{\nu}( \Delta u) \Vert_{L^2(\partial \Omega_{-},\ve)} \\ 
%&  + \Vert \sqrt{-(\A\cdot \nu)} \ e^{-\frac{x\cdot \A}{h}} \partial_{\nu}u \Vert_{L^2(\partial \Omega_{-},\ve)}\Big{)}.
\end{aligned}
\notag
\end{equation}
and therefore
\begin{equation}
 \begin{aligned}
  \Vert \sqrt{\A\cdot \nu}\ e^{-\frac{x\cdot \A}{h}} \partial_{\nu}u \Vert_{L^2(\partial \Omega_{+,\ve})} \leq 
  & C\Big{(}h^{\frac{3}{2}} \Vert e^{-\frac{x\cdot \A}{h}} \mathcal{B}_{q_{1}}u  \Vert_{L^{2}(\Omega)}\\&+ h \Vert \sqrt{-\A\cdot \nu}\  e^{-\frac{x\cdot \A}{h}} \partial_{\nu}( \Delta u) \Vert_{L^2(\partial \Omega_{-},\ve)} \\ 
  &  + \Vert \sqrt{-(\A\cdot \nu)} \ e^{-\frac{x\cdot \A}{h}} \partial_{\nu}u \Vert_{L^2(\partial \Omega_{-},\ve)}\Big{)}.
 \end{aligned}
\notag
\end{equation}
We then have
\begin{equation}
\begin{aligned}
\Vert e^{-\frac{x\cdot \A}{h}} \partial_{\nu}u \Vert_{L^2(\partial \Omega_{+,\epsilon})} &\leq \frac{C}{\sqrt{\epsilon}} (h^\frac{3}{2} \Vert e^{-\frac{x\cdot \A}{h}}\mathcal{B}_{q_{1}}u \Vert_{L^2(\Omega)} \\& +h \Vert \sqrt{-(\A\cdot \nu)} \  e^{-\frac{x\cdot \A}{h}} \partial_{\nu}(\Delta u)  \Vert_{L^2(\partial \Omega_{-},\ve)} \\ 
& + \Vert \sqrt{-(\A\cdot \nu)} \ e^{-\frac{x\cdot \A}{h}} \partial_{\nu}u \Vert_{L^2(\partial \Omega_{-},\ve)}) \\
& \leq \frac{C}{\sqrt{\epsilon}} h^\frac{3}{2} \Vert e^{-\frac{x\cdot \A}{h}}(\mathcal{B}_{q_{1}})u \Vert_{L^2(\Omega)}  \\& +\frac{hC}{\sqrt{\epsilon}} \sqrt{- \displaystyle \inf_{\partial \Omega_{-}} (\A\cdot \nu)} \Vert e^{-\frac{x\cdot \A}{h}} \partial_{\nu}(\Delta u) \Vert_{L^2(\partial \Omega_{-},\epsilon)} \\
&  + \frac{C}{\sqrt{\epsilon}} \sqrt{- \inf_{\partial\Omega_{-}} \A\cdot \nu} \Vert e^{-\frac{x\cdot \A}{h}} \partial_{\nu} u \Vert_{L^2(\partial \Omega_{-},\epsilon)}.
\end{aligned}
\notag
\end{equation}
Similar to the previous estimate, we have
\begin{equation}
 \begin{aligned}
\vert \int_{\partial \Omega_{+,\epsilon}} \NT\NT\NT\NT\NT\NT\partial_{\nu}(u_{1}-u_{2})\overline{(\Delta v)}\ \D S \vert  
& \leq \frac{C}{\sqrt{\epsilon}}h^{\frac{3}{2}} \Vert e^{-\frac{x\cdot \A}{h}} (\mathcal{B}_{q_{1}})(u_{1}-u_{2}) \Vert_{L^{2}(\Omega)} \\
&+\frac{h C}{\sqrt{\epsilon}} \sqrt{-\displaystyle \inf_{\partial \Omega_{-}} \A\cdot \nu} \Vert e^{-\frac{x\cdot \A}{h}} \partial_{\nu}(\Delta(u_{1}-u_{2})) \Vert_{L^{2}(\partial \Omega_{-,\epsilon})} \\
&+\frac{C}{\sqrt{\epsilon}} \sqrt{-\displaystyle \inf_{\partial \Omega_{-}} (\A\cdot \nu)} \ \Vert e^{-\frac{x\cdot \A}{h}} \partial_{\nu}(u_{1}-u_{2}) \Vert_{L^{2}(\partial \Omega_{-,\epsilon})} \\
& \leq C h^{\frac{3}{2}} + \frac{C}{h^{4}}e^{\frac{3R}{h}} \Vert \wt{\mathcal{N}}_{q_{1}}-\wt{\mathcal{N}}_{q_{2}} \Vert\\& 
  \leq C h^{\frac{3}{2}} + C e^{\frac{8R}{h}} \Vert \wt{\mathcal{N}}_{q_{1}}-\wt{\mathcal{N}}_{q_{2}} \Vert,
 \end{aligned}
\notag
\end{equation}
where the constant $C$ again depends upon $\ve$.\\
Therefore using the estimates obtained above, we have
 \begin{align} \label{Partial data partial FT estimate}
  \notag\vert \int_{\Omega} e^{-i x\cdot\xi} (q_{2}-q_{1}) \ dx \vert  & \leq \vert\NT\NT \int_{\partial \Omega_{+,\epsilon}}\NT\NT\NT\NT\NT \partial_{\nu}(\Delta(u_{1}-u_{2}))\bar{v} \ \D S \vert + \vert \NT\NT\int_{\partial \Omega_{+,\epsilon}}\NT\NT\NT\NT \partial_{\nu}(u_{1}-u_{2})\overline{(\Delta v)}\ \D S \vert \\
 \notag &  + \vert\NT\NT \int_{\partial \Omega_{-,\epsilon}}\NT\NT\NT\NT \partial_{\nu}(\Delta(u_{1}-u_{2})) \bar{v} \ \D S + \int_{\partial \Omega_{-,\epsilon}}\NT\NT\NT\NT \partial_{\nu}(u_{1}-u_{2}) \overline{(\Delta v)}\ \D S \vert \\
  \notag &  + \vert \int_{\Omega} (q_{2}-q_{1}) e^{-i x.\xi} (h \bar{r}_{1}+h r_{2}+h^{2} \bar{r}_{1} r_{2}) \ dx \vert \\
  \notag & \leq C \sqrt{h}+C e^{\frac{8R}{h}} \Vert \wt{\mathcal{N}}_{q_{1}}-\wt{\mathcal{N}}_{q_{2}} \Vert+C h^{\frac{3}{2}}\\\notag & +C e^{\frac{8R}{h}} \Vert \wt{\mathcal{N}}_{q_{1}}-\wt{\mathcal{N}}_{q_{2}} \Vert \\ 
  \notag &  +C e^{\frac{9R}{h}} \Vert \wt{\mathcal{N}}_{q_{1}}-\wt{\mathcal{N}}_{q_{2}} \Vert+ C h \\
  & \leq C (\sqrt{h} + e^{\frac{9R}{h}} \Vert \wt{\mathcal{N}}_{q_{1}}-\wt{\mathcal{N}}_{q_{2}}  \Vert).
 \end{align}
 The argument that now follows is similar to the one in \cite{Heck-Wang-StabilityPaper}. We will apply Vessella's result given in Theorem \ref{Vessella Result} for the following set up. We take $D$ to be the ball $B(0,2)$ and $E =V\cap B(0,1)$ where $V$ is a suitable small open cone centered at $0$ obtained by perturbing the vector $\A$ slightly and recalling that $\xi$ is perpendicular to $\A$.
 Note that the above estimate is valid for all $\xi \in V $ such that $\vert \xi \vert < \frac{2}{h} $.

 Now let $q=q_{1}-q_{2}$ extended to $\Rb^{n}$ as $0$ outside $\O$ and for a fixed $\rho \in (0,\frac{2}{h})$, let $f(\xi)=\wh{q}(\rho \xi)$. Then $f$ is analytic in $B(0,2)$ and 
 \[
 \vert D^{\alpha} f(\xi) \vert \leq \Vert q \Vert_{L^{1}(\Omega)} \frac{\rho^{\vert \alpha \vert}}{(\text{diam}(\Omega)^{-1})^{\vert \alpha \vert}} \leq 2M \lvert \O\rvert \alpha ! \frac{e^{n \rho}}{(\text{diam}(\Omega)^{-1})^{\vert \alpha \vert}}.
 \]
 Taking $C$ and $\lambda$ in Vessella's result to be $C=2M\lvert \O\rvert e^{n\rho}$ and $\lambda=\text{diam}(\Omega)^{-1}$, we get that there exists a constant $\g_{1}\in (0,1)$ such that 
 \[
|f(\xi)|\leq C^{1-\g_{1}(|E|/|D|)}\lb \sup\limits_{E} |f(\xi)|\rb^{\g_{1}(|E|/|D|)}, \mbox{ for all } \xi\in B(0,1).
\]
Letting $\theta=\g_{1}|E|/|D|$, we have that for all $|\xi|<\rho$, 
\Beq\label{FT estimate on the unit ball}
|\wh{q}(\xi)|\leq C^{1-\theta}\lb \sup\limits_{V\cap B(0,\rho)} |\wh{q}(\xi/\rho)|\rb^{\theta}.
\Eeq
Note that the constant $\theta$ is independent of $\rho$ and $h$.
We have 
 \begin{align*}
 \Vert q \Vert_{H^{-1}(\mathbb{R}^{n})}^{\frac{2}{\theta}} & = \lb \int_{\vert \xi \vert < \rho} \frac{\vert \wh{q} (\xi)\vert^{2}}{1+\vert \xi \vert^{2}} \D \xi+\int_{\vert \xi \vert\geq \rho } \frac{\vert \wh{q}(\xi) \vert^{2}}{1+\vert \xi \vert^{2}} \D \xi \rb^{\frac{1}{\theta}} \\
& \leq C \lb\rho^{\frac{n}{\theta}} \Vert \wh{q} \Vert_{L^{\infty}(B(0,\rho))}^{\frac{2}{\theta}} +\frac{1}{\rho^{\frac{2}{\theta}}} \rb.\\
\intertext{The estimate of the second term on the right hand side above is obtained from Plancherel identity. Now from \eqref{FT estimate on the unit ball}, it follows that the left hand side is}
& \leq C \lb\rho^{\frac{n}{\theta}} e^{2n \rho \frac{1-\theta}{\theta}} \Vert \wh{q} (\xi) \Vert^{2}_{L^{\infty}(V \cap B(0,\rho))}+\frac{1}{\rho^{\frac{2}{\theta}}}\rb\\
& \leq C \lb\rho^{\frac{n}{\theta}} e^{2n \rho \frac{1-\theta}{\theta}} e^{\frac{18 R}{h}} \Vert \wt{\mathcal{N}}_{q_{1}}-\wt{\mathcal{N}}_{q_{2}} \Vert^{2}+\rho^{\frac{n}{\theta}} e^{2n \rho \frac{1-\theta}{\theta}} h +\frac{1}{\rho^{\frac{2}{\theta}}}\rb.
 \end{align*}

Let $L=\frac{3n+2-2n \theta}{\theta} $ and $\delta=e^{-{e^{K/{h_{0}^{\frac{1}{L}}}}}},$ where $K= \frac{2n+2}{\theta}+4n \frac{1-\theta}{\theta}+18 R.$\\
Let $\Vert \wt{\mathcal{N}}_{q_{1}}-\wt{\mathcal{N}}_{q_{2}} \Vert < \delta $.\\
Then choose $\rho=\frac{1}{K} \ln \vert \ln \Vert \wt{\mathcal{N}}_{q_{1}}-\wt{\mathcal{N}}_{q_{2}} \Vert \vert$ and 
\[
h=\frac{1}{\rho^{\frac{n+2}{\theta}}e^{\frac{2n\rho(1-\theta)}{\theta}}}.
\]
\textit{Claim 1:} $\rho < \frac{2}{h}$. We have 
$\rho h=\frac{\rho}{\rho^{\frac{n+2}{\theta}} e^{\frac{2n \rho (1-\theta)}{\theta}}} = \frac{1}{\rho^{\frac{n+2}{\theta}-1} e^{\frac{2n \rho (1-\theta)}{\theta}}} \leq \frac{1}{\rho^{\frac{n+2}{\theta}-1}}$.\\
Now since $n \geq 3$ we have $\frac{n+2}{\theta}-1 > 4$, and hence $\rho^{\frac{n+2}{\theta}-1}> \rho^4 $.
Therefore, $\rho h < \frac{1}{\rho^4} < 2 \ \ (\text{since} \ \rho > 1 ) $.

\textit{Claim 2:} $h < h_{0}$. We have 
\begin{equation}
 \begin{aligned}
  \Vert \wt{\mathcal{N}}_{q_{1}}-\wt{\mathcal{N}}_{q_{2}}  \Vert &< e^{-{e^{K/{h_{0}^{\frac{1}{L}}}}}}\ll 1 \\
  \Rightarrow \ln \Vert \wt{\mathcal{N}}_{q_{1}}-\wt{\mathcal{N}}_{q_{2}}  \Vert &< -{e^{K/{h_{0}^{\frac{1}{L}}}}} \\
  \Rightarrow \vert \ln \Vert \wt{\mathcal{N}}_{q_{1}}-\wt{\mathcal{N}}_{q_{2}}  \Vert \vert &>  e^{K/{h_{0}^{\frac{1}{L}}}}\\
  \Rightarrow \frac{1}{K} \ln \vert \ln \Vert \wt{\mathcal{N}}_{q_{1}}-\wt{\mathcal{N}}_{q_{2}}  \Vert \vert &> \frac{1}{h_{0}^{\frac{1}{L}}} \\
  \Rightarrow \rho^{L} &> \frac{1}{h_{0}} \\
  \Rightarrow \rho^{\frac{3n+2-2n \theta}{\theta}} &> \frac{1}{h_{0}} \\
  \Rightarrow \frac{1}{\rho^{\frac{n+2}{\theta}} \rho^{\frac{2n(1-\theta)}{\theta}}} &< h_{0} \\
  \Rightarrow h=\frac{1}{\rho^{\frac{n+2}{\theta}}e^{\frac{2n\rho(1-\theta)}{\theta}}} \leq \frac{1}{\rho^{\frac{n+2}{\theta}} \rho^{\frac{2n(1-\theta)}{\theta}}} &< h_{0}
 \end{aligned}
\notag
\end{equation}

\textit{Claim 3:} $1-h^{2}\frac{\vert \xi \vert^{2}}{4} >0$. This is because 
$h^{2} \frac{\vert \xi \vert^{2}}{4} \leq h^{2} \frac{\rho^{2}}{4}= \frac{\rho^{2}}{4 \rho^{\frac{2(n+2)}{\theta}} e^{4n \rho \frac{(1-\theta)}{\theta}}} =\frac{\rho^{2-\frac{2n+4}{\theta}}}{4 e^{4n \rho \frac{1-\theta}{\theta}}} < \rho^{2-\frac{2n+4}{\theta}}=\frac{1}{\rho^{\frac{2n+4}{\theta}-2}}$.
Now $\frac{2n+4}{\theta}-2  > 8$ since $n \geq 3$ and therefore $\rho^{\frac{2n+4}{\theta}-2} \geq \rho^{8} $ which in turn implies that
$$h^{2} \frac{\vert \xi \vert^{2}}{4} < \frac{1}{\rho^{8}} <1 .$$

 Then we have 
\begin{align*}
\rho^{\frac{n}{\theta}} e^{2n \rho \frac{1-\theta}{\theta}+ \frac{18 R}{h}}&= \rho^{\frac{n}{\theta}} e^{[2n \rho \frac{1-\theta}{\theta}+ 18 R (\rho^{\frac{n+2}{\theta}} e^{2n \rho \frac{1-\theta}{\theta}}) ]}\\
&\leq e^{\frac{n}{\theta}\rho+2n \rho \frac{1-\theta}{\theta}+18 R \rho^{\frac{n+2}{\theta}} e^{2n \rho \frac{1-\theta}{\theta}}} \text{ since } \rho^{\frac{n}{\theta}}\leq e^{\frac{n}{\theta}\rho}\\
&\leq e^{e^{\lb\frac{n}{\theta} \rho+ 2n \rho \frac{1-\theta}{\theta}\rb}+e^{\lb18 R+\frac{n+2}{\theta} \rho+2n \rho \frac{1-\theta}{\theta}\rb}} \\
&\leq e^{e^{\lb\frac{n}{\theta} \rho+ 2n \rho \frac{1-\theta}{\theta}\rb}+e^{\lb 18 R \rho+\frac{n+2}{\theta} \rho+2n \rho \frac{1-\theta}{\theta}\rb}} \text{ since } \rho\geq 1\\
& \leq  C e^{e^{\lb\frac{n}{\theta}+2n \frac{1-\theta}{\theta}+18 R+\frac{n+2}{\theta}+2n \frac{1-\theta}{\theta}\rb\rho}} \text {since } e^{a}+e^{b}\leq 1+e^{a+b}.
\end{align*}
\begin{comment}
But 
\begin{equation}
 \begin{aligned}
  \rho^{\frac{n}{\theta}} e^{[2n \rho \frac{1-\theta}{\theta}+ 18 R (\rho^{\frac{n+2}{\theta}} e^{2n \rho \frac{1-\theta}{\theta}}) ]} & \leq 
  e^{\frac{n}{\theta}\rho+2n \rho \frac{1-\theta}{\theta}+18 R \rho^{\frac{n+2}{\theta}} e^{2n \rho \frac{1-\theta}{\theta}}} \\
 & \leq e^{e^{[\frac{n}{\theta} \rho+ 2n \rho \frac{1-\theta}{\theta}]}+e^{[18 R+\frac{n+2}{\theta} \rho+2n \rho \frac{1-\theta}{\theta}]}} \\
 & \leq e^{e^{[\frac{n}{\theta} \rho+ 2n \rho \frac{1-\theta}{\theta}]}+e^{[18 R \rho+\frac{n+2}{\theta} \rho+2n \rho \frac{1-\theta}{\theta}]}} \\
 & \leq C \ e^{e^{[\frac{n}{\theta}+2n \frac{1-\theta}{\theta}+18 R+\frac{n+2}{\theta}+2n \frac{1-\theta}{\theta}]\rho}}
 \end{aligned}
\notag
\end{equation}
\end{comment}

Therefore since $K= \frac{2n+2}{\theta}+4n \frac{1-\theta}{\theta}+18 R $ and $\rho=\frac{1}{K} \ln \vert \ln \Vert \wt{\mathcal{N}}_{q_{1}}-\wt{\mathcal{N}}_{q_{2}} \Vert \vert $, we obtain
 \[
\Vert q \Vert^{\frac{2}{\theta}}_{H^{-1}(\Omega)} \leq C (\Vert \wt{\mathcal{N}}_{q_{1}}-\wt{\mathcal{N}}_{q_{2}} \Vert +\ (\frac{1}{K} \ln \vert \ln \Vert \wt{\mathcal{N}}_{q_{1}}-\wt{\mathcal{N}}_{q_{2}} \Vert \vert)^{-\frac{2}{\theta}})
\]
and hence 
\[ \Vert q \Vert_{H^{-1}(\Omega)} \leq C (\Vert \wt{\mathcal{N}}_{q_{1}}-\wt{\mathcal{N}}_{q_{2}} \Vert  +\ (\frac{1}{K} \ln \vert \ln \Vert \wt{\mathcal{N}}_{q_{1}}-\wt{\mathcal{N}}_{q_{2}} \Vert \vert)^{-\frac{2}{\theta}})^{\frac{\theta}{2}},
 \]
whenever $\Vert \wt{\mathcal{N}}_{q_{1}}-\wt{\mathcal{N}}_{q_{2}} \Vert < \delta $. \\
When $\Vert \wt{\mathcal{N}}_{q_{1}}-\wt{\mathcal{N}}_{q_{2}} \Vert \geq \delta$,
we have $$\Vert q_{1}-q_{2} \Vert_{H^{-1}(\Omega)} \leq C \Vert q_{1}-q_{2} \Vert_{L^{\infty}(\Omega)} \leq \frac{2CM}{\delta^{\frac{\theta}{2}}} \delta^{\frac{\theta}{2}} \leq \frac{2CM}{\delta^{\frac{\theta}{2}}} \Vert \wt{\mathcal{N}}_{q_{1}}-\wt{\mathcal{N}}_{q_{2}} \Vert^{\frac{\theta}{2}} $$ and the desired estimate follows.
\end{proof}
\appendix
\section{}\label{PartialData-Identifiability}
In this section we prove the unique determination of $q$ from \eqref{Biharmonic-NavierBVP} when the Neumann data $\wt{N}_{q}$ is known on slightly more than half the boundary. This is already done in a more general set-up with limiting Carleman weights in \cite{KLU-BiharmonicPaper}, where the authors use logarithmic weights.  We give here the proof with linear Carleman weight following \cite{Bukhgeim-Uhlmann-InverseProblemPaper} for the sake of completeness.

%stability estimates when the Neumann data is measured on an open subset of the boundary that contains slightly more than half of the boundary. 
\begin{theorem}\cite{KLU-BiharmonicPaper} Let $\O\subset \Rb^{n}, n\geq 3$ be a bounded domain with smooth boundary. Consider Equation \eqref{Biharmonic-NavierBVP} for two potentials $q_{1}, q_{2} \in L^{\infty}(\O)$.
Let $\wt{\Nc}_{q_{1}}$ and $\wt{\Nc}_{q_{2}}$ be the corresponding Dirichlet-to-Neumann maps measured on $\PD \O_{-,\ve}$. If $\wt{N}_{q_{1}}=\wt{N}_{q_{2}}$, then $q_{1}=q_{2}$.
\end{theorem}

\bpr As before, we start with the following integral identity.
$$\int_{\Omega} (q_{2}-q_{1})u_{2} \bar v \ dx  =-\int_{\partial \Omega_{+,\epsilon}} \partial_{\nu}(-\Delta(u_{1}-u_{2}))\bar v \ \D S -\int_{\partial \Omega_{+,\epsilon}} \partial_{\nu}(u_{1}-u_{2}) \overline{(-\Delta v)}\ \D S $$ 
\begin{equation}
\begin{aligned}
\vert \int_{\partial \Omega_{+,\epsilon}} \partial_{\nu}(\Delta(u_{1}-u_{2}))\bar v \ \D S \vert &= \vert \int_{\partial \Omega_{+,\epsilon}} e^{-\frac{x\cdot \A}{h}} \partial_{\nu}(\Delta(u_{1}-u_{2})) e^{\frac{x\cdot \A}{h}} \bar v \ \D S \vert \\
&\leq \Vert e^{-\frac{x\cdot \A}{h}} \partial_{\nu} (\Delta(u_{1}-u_{2})) \Vert_{L^2(\partial \Omega_{+,\epsilon})} \Vert e^{\frac{x\cdot \A}{h}} \bar v \Vert_{L^2(\partial \Omega_{+,\epsilon})}
\end{aligned}
\notag
\end{equation}
\begin{equation}
\begin{aligned}
\vert \int_{\partial \Omega_{+,\epsilon}} \partial_{\nu} (u_{1}-u_{2}) \overline{(\Delta v)} \ \D S \vert &= \vert \int_{\partial \Omega_{+,\epsilon}} e^{-\frac{x\cdot \A}{h}} \partial_{\nu} (u_{1}-u_{2}) e^{\frac{x\cdot \A}{h}} \overline{(\Delta v)} \ \D S \vert \\
& \leq \Vert e^{-\frac{x\cdot \A}{h}} \partial_{\nu} (u_{1}-u_{2}) \Vert_{L^2(\partial \Omega_{+,\epsilon})} \Vert e^{\frac{x\cdot \A}{h}} \overline{(\Delta v)} \Vert_{L^2(\partial \Omega_{+,\epsilon})} 
\end{aligned}
\notag
\end{equation}
From the boundary Carleman estimate, we have 
$$\Vert \sqrt{\A\cdot \nu} \ e^{-\frac{x\cdot \A}{h}}\ \partial_{\nu}(-h^2 \Delta u) \Vert_{L^2(\partial \Omega_{+})} \leq \frac{C}{h^{\frac{3}{2}}} \Vert e^{-\frac{x\cdot \A}{h}} (h^4 \mathcal{B}_{q_{1}})u \Vert_{L^2(\Omega)} $$
where $u=u_{1}-u_{2}$.\\
Using this we get
\begin{equation}
\begin{aligned}
\sqrt{\epsilon}  \Vert e^{-\frac{x\cdot \A}{h}} \partial_{\nu}(-h^2 \Delta u) \Vert_{L^2(\partial \Omega_{+,\epsilon})} &\leq \Vert \sqrt{\A\cdot \nu}\ e^{-\frac{x\cdot \A}{h}} \partial_{\nu}(-h^2 \Delta u) \Vert_{L^2(\partial \Omega_{+,\epsilon})} \\
& \leq \Vert \sqrt{\A\cdot \nu}\ e^{-\frac{x\cdot \A}{h}} \partial_{\nu}(-h^2 \Delta u) \Vert_{L^2(\partial \Omega{+})} \\
& \leq \frac{C}{h^\frac{3}{2}} \Vert e^{-\frac{x\cdot \A}{h}} (h^4 \mathcal{B}_{q_{1}})u \Vert_{L^2(\Omega)}.
\end{aligned}
\notag
\end{equation}
Therefore 
\begin{equation}
\begin{aligned}
\Vert  e^{-\frac{x\cdot \A}{h}} \partial_{\nu}(\Delta (u_{1}-u_{2})) \Vert_{L^2(\partial \Omega_{+,\epsilon})} & \leq C \sqrt{\frac{h}{\epsilon}} \Vert e^{-\frac{x\cdot \A}{h}}(q_{2}-q_{1})e^{\frac{\I x.\zeta_{2}}{h}}(1+hr_{2}) \Vert_{L^2(\Omega)} \\
& \leq C \sqrt{\frac{h}{\epsilon}} \Vert (1+h r_{2}) \Vert_{L^2(\Omega)}  \leq C \sqrt{\frac{h}{\epsilon}}.
\end{aligned}
\notag
\end{equation}
From the boundary Carleman estimate, we also have
$$\Vert \sqrt{\A\cdot \nu}\ e^{-\frac{x\cdot \A}{h}} \partial_{\nu}u \Vert_{L^2(\partial \Omega_{+})} \leq \frac{C}{h^\frac{5}{2}} \Vert    e^{-\frac{x\cdot \A}{h}} (h^4 \mathcal{B}_{q_{1}})u \Vert_{L^2(\Omega)} $$
Again using this, we get
\begin{equation}
\begin{aligned}
\sqrt{\epsilon} \Vert e^{-\frac{x\cdot \A}{h}} \partial_{\nu}u \Vert_{L^2(\partial \Omega_{+,\epsilon})} & \leq \Vert \sqrt{\A\cdot \nu} \ e^{-\frac{x\cdot \A}{h}} \partial_{\nu}u \Vert_{L^2(\partial \Omega_{+,\epsilon})} \\
& \leq \Vert \sqrt{\A\cdot \nu} \ e^{-\frac{x\cdot \A}{h}} \partial_{\nu}u \Vert_{L^2(\partial \Omega_{+})} \\
& \leq \frac{C}{h^\frac{5}{2}} \Vert e^{-\frac{x\cdot \A}{h}} (h^4 \mathcal{B}_{q_{1}})u  \Vert_{L^2(\Omega)}
\end{aligned}
\notag
\end{equation}
\begin{equation}
\begin{aligned}
\Rightarrow \Vert e^{-\frac{x\cdot \A}{h}} \partial_{\nu}(u_{1}-u_{2})\Vert_{L^2(\partial \Omega_{+,\epsilon})} & \leq \frac{C}{h^\frac{5}{2} \sqrt{\epsilon}} \Vert  e^{-\frac{x\cdot \A}{h}} (h^4 \mathcal{B}_{q_{1}})u \Vert_{L^2(\Omega)} \\
& \leq \frac{C h^\frac{3}{2}}{\sqrt{\epsilon}} \Vert (1+hr_{2}) \Vert_{L^2(\Omega)}   \leq \frac{C h^\frac{3}{2}}{\sqrt{\epsilon}}
\end{aligned}
\notag
\end{equation}
Next we show that the terms $\Vert e^{\frac{x\cdot \A}{h}} \bar v \Vert_{L^2(\partial \Omega_{+,\epsilon})}$ and $\Vert e^{\frac{x\cdot \A}{h}} \overline{(\Delta v)} \Vert_{L^2(\partial \Omega_{+,\epsilon})}$ are bounded.
The term
\begin{equation}
\begin{aligned}
\Vert e^{\frac{x\cdot \A}{h}} \bar v \Vert_{L^2(\PD\Omega_{+,\epsilon})} & \leq \Vert e^{\frac{x\cdot \A}{h}} \bar v \Vert_{L^2(\partial \Omega)} = \Vert (1+h \overline{r_{1}})\Vert_{L^2(\partial \Omega)} \\
& \leq C(1+ \Vert h r_{1} \Vert_{H^1(\Omega)}) \leq C,
\end{aligned}
\notag
\end{equation}
since $h \ll 1 $.\\
Again
\[
e^{\frac{x\cdot \A}{h}} \overline{(\Delta v)}= e^{\frac{\I}{h}(-\frac{h}{2}(x.\xi)-\sqrt{1-h^2 \frac{\vert \xi \vert^2}{4}}(x\cdot\beta))} h \Delta \overline{r_{1}}-2\I e^{\frac{\I}{h}(-\frac{h}{2}(x\cdot\xi)-\sqrt{1-h^2 \frac{\vert \xi \vert^2}{4}}(x\cdot\beta))} \overline{\zeta_{1}}\cdot\nabla \overline{r_{1}}. 
\] 
Therefore, we have
\begin{equation}
\begin{aligned}
\Vert e^{\frac{x\cdot \A}{h}} \overline{(\Delta v)} \Vert_{L^2(\partial \Omega_{+,\epsilon})} & \leq \Vert  e^{\frac{x\cdot \A}{h}} \overline{(\Delta v)} \Vert_{L^2(\Omega)} \\
& \leq \Vert h \Delta r_{1} \Vert_{L^2(\partial \Omega)} + C \Vert \nabla r_{1} \Vert_{L^2(\partial \Omega)} \\
& \leq C(h \Vert \Delta r_{1} \Vert_{H^1(\Omega)} + \Vert  \nabla r_{1} \Vert_{H^1(\Omega)}) \\
& \leq C(\frac{h}{h^2} \Vert r_{1} \Vert_{H^4_{\mathrm{scl}}(\Omega)} + \frac{1}{h} \Vert r_{1} \Vert_{H^4_{\mathrm{scl}}(\Omega)}) \leq C.
\end{aligned}
\notag
\end{equation}
%since $h \ll 1$.\\
%
Also using the estimates on $r_{1}, r_{2} $ it follows that as limit $h \rightarrow 0 $, 
\begin{equation}
\int_{\Omega} (q_{2}-q_{1})u_{2} \bar{v} \ dx \rightarrow \int_{\Omega} e^{-\I x\cdot\xi} (q_{2}-q_{1})\: \D x.
  \notag
\end{equation}
Therefore combining all the above estimates and passing to the limit as $h \rightarrow 0 $, we have
\begin{equation}
\int_{\Omega} e^{-\I x\cdot\xi}(q_{2}-q_{1}) \: \D x =0
 \notag
\end{equation}
for all $\xi \in \mathbb{R}^{n} $ perpendicular to $\alpha$. Varying $\alpha $ in a sufficiently small neighborhood, we see that above estimates
is true for all $\xi $ in an open cone in $\mathbb{R}^{n} $. A simple application of the Paley-Wiener theorem then implies that $q_{2}=q_{1}$ on $\O$. This concludes the proof.
\epr

%

%\bibliographystyle{plain}
%\bibliography{references.bib}
%\end{document}



%% References





%\begin{thebibliography}{00}



 %\end{thebibliography}

\end{document}